\documentclass[a4paper]{article}

\usepackage{amsmath,amsfonts,amssymb,amsthm,amscd,enumerate}
\usepackage[all]{xy}

 \newtheorem{thm}{Theorem}
 
 \newtheorem{lem}[thm]{Lemma}

 \theoremstyle{definition}
 \newtheorem{defn}[thm]{Definition}
 \theoremstyle{remark}

\renewcommand{\[}{\begin{equation}}
\renewcommand{\]}{\end{equation}}

 \newcommand{\C}{\mathbb{C}}
 \newcommand{\N}{\mathbb{N}}
 \newcommand{\Z}{\mathbb{Z}}
 \newcommand{\D}{\mathbb{D}}

 \newcommand{\A}{\mathcal{A}}
  
 \newcommand{\hH}{\mathcal{H}}
 \newcommand{\T}{\mathcal{T}} 
  \newcommand{\F}{\mathcal{F}} 
 
\newcommand{\hsp}{{\hspace{-1pt}}}
\newcommand{\hs}{{\hspace{1pt}}}
 \newcommand{\half}{\mbox{$\frac{1}{2}$}}
\newcommand{\lra}{\longrightarrow}

\newcommand{\ra}{\rightarrow}

\newcommand{\im}{\mathrm{i}}
\newcommand{\id}{\mathrm{id}}
\newcommand{\e}{\mathrm{e}}

\newcommand{\lN}{\ell_2(\N)}

\def\bD{\bar\D}
\def\LD{L_2(\D)}
\def\AD{A_2(\D)}
\def\KAD{K(A_2(\D))}
\def\KlN{K(\lN)}

\def\CD{C(\bar\D)}
\def\dom{\mathrm{dom}}
\def\spec{\mathrm{spec}}
\def\lin{\mathrm{span}}
\def\End{\mathrm{End}}
\def\ind{\mathrm{ind}}
\def\cspan{\overline{\mathrm{span}}}
\def\ddz{\mbox{$\frac{\partial}{\partial z}$}}
\def\ddbz{\mbox{$\frac{\partial}{\partial \bar z}$}}
\def\dz{\mbox{$\partial_z$}}
\def\dzs{\mbox{$\partial_z^{\hs*}$}}

\def\bS{\mathbb{S}^1} 
\def\bT{\mathbb{T}^2} 
\def\CbT{C(\mathbb{T}^2)} 
\def\CTq{C(\mathbb{T}^2_q)}
\def\s{\sigma} 
\def\a{\alpha} 
\def\ker{\mathrm{ker}} 
\def\CS{C(\mathbb{S}^1)}

\newcommand{\msum}[2]{\underset{#1}{\overset{#2}{\mbox{$\sum$}}}}

\title{Dirac operator on a noncommutative Toeplitz torus}

\author{{\sc Fredy D\'iaz Garc\'ia} \\
\normalsize
Instituto de F\'isica y Matem\'aticas\\
\normalsize
Universidad Michoacana de San Nicol\'as de Hidalgo, Morelia, M\'exico\\[2pt]
\normalsize
and\\[2pt]
\normalsize
Centro de Ciencias Matem\'aticas, Campus Morelia\\
\normalsize
Universidad Nacional Aut\'onoma de M\'exico (UNAM), Morelia, M\'exico\\
\normalsize
e-mail: {\it lenonndiaz@gmail.com}\\[16pt] 
{\sc Elmar Wagner\footnote{
corresponding author \ \  {\it MSC2010:} 58B34; 46L87.  \ \  
}
} \\
\normalsize
Instituto de F\'isica y Matem\'aticas\\
\normalsize
Universidad Michoacana de San Nicol\'as de Hidalgo, Morelia, M\'exico\\
\normalsize
e-mail: {\it elmar@ifm.umich.mx}}

\usepackage{geometry}
\date{}                                       

\begin{document}
\maketitle

\begin{abstract}
We construct a $1^+$-summable regular even spectral triple 
for a noncommutative torus 
defined by a C*-subalgebra of the Toeplitz algebra. 
\end{abstract}

\maketitle

\section{Introduction} 
In noncommutative geometry \cite{C}, a noncommutative topological space is presented 
by a noncommutative C*-al\-ge\-bra. 
Usually definitions of such \mbox{C*-al}\-ge\-bras  
are motivated by imitating some features of the classical spaces. 
For instance, 
a noncommutative version of any compact two-dimensional surface 
without boundary can be found in \cite{W}, where 
the corresponding \mbox{C*-al}\-ge\-bras are defined as 
subalgebras of the Toeplitz algebra.  

The metric aspects of a noncommutative space are captured by the notation 
of a spectral triple \cite{CL}. 
Given a unital C*-algebra $A$,  
a spectral triple $(\A,\hH, D)$ for $A$ consists of a dense *-sub\-al\-gebra $\A\subset A$, 
a Hilbert space $\hH$ together with a faithful *-representation $\pi:\A\ra B(\hH)$, and a 
self-adjoint operator $D$ on $\hH$, called \emph{Dirac operator}, such that 
\begin{align}
&[D, \pi(a)] \in B(\hH)\ \ \text{for \,all} \ \ a\in\A, \\
& (D+\im)^{-1} \in K(\hH). 
\end{align}
Here $K(\hH)$ denotes the set of compact operators on $\hH$. 

The purpose of the present paper is the construction of a spectral triple for 
the noncommutative torus from \cite{W}. The noncommutative torus was chosen because 
the self-adjoint operator $D$ from the spectral triple 
has a similar structure to the Dirac operator on a classical torus with 
a flat metric. Our main theorem shows that this spectral triple is even, regular, and $1^+$-summable. 

For the convenience of the reader, we recall the definitions of 
the just mentioned properties of a spectral triple (see \cite{GFV}). 
By a slight abuse of notation, 
we will not distinguish between a densely defined closable operator and its closure. 
A spectral triple $(\A,\hH, D)$ is said to be \emph{even}, 
if there exists a grading operator $\gamma \in B(\hH)$ 
satisfying 
\[ \label{gam} 
\gamma^* = \gamma, \quad \gamma^2=1, \quad \gamma D=-D\gamma, \quad 
\gamma \pi(a) =\pi(a) \gamma \  \text{ for \,all } \  a\in\A. 
\] 
We call $(\A,\hH, D)$ \emph{regular}, if $\delta^k(a)\in B(\hH)$  and $\delta^k([D,a])\in B(\hH)$ 
for all $a\in \A$ and $k\in \N$, where $\delta(x):= [|D|,x]$ for $x\in B(\hH)$. 
The term \emph{$1^+$-summ\-able} means that $(1+ |D|)^{-(1+\epsilon)}$ is 
a trace class operator for all $\epsilon >0$ but $(1+ |D|)^{-1}$ is not a trace class operator.  

Consider the polar decomposition $D=F\hs |D|$ of the Dirac operator. 
The grading operator $\gamma$ gives rise to a decomposition $\hH=\hH_+\oplus \hH_-$ 
such that $\gamma=\begin{pmatrix} 1 & 0 \\0 & -1 \end{pmatrix}$ 
and $F= \begin{pmatrix} 0 & F_{+-} \\ F_{-+} & 0 \end{pmatrix}$. 
If the spectral triple satisfies the properties of the previous paragraph, then 
$F_{+-}$ and $F_{-+} $ are Fredholm operators and one defines $\ind(D):=\ind(F_{+-})$. 
The operator $F$ is called the fundamental class of $D$ and it is said to be non-trivial if 
$\ind(D)\neq 0$.

\section{Noncommutative Toeplitz torus}    \label{sec-1} 

Let $\D:= \{ z\in \C: |z|<1\}$ be the open unit disc and $\bD:= \{ z\in \C: |z|\leq 1\}$ 
its closure in $\C$. Consider the Hilbert space $\LD$ 
with respect to the standard Lebesgue measure 
and its closed subspace $\AD$ consisting of 
all $L_2$-functions which are holomorphic in $\D$. 
We denote by $P$ the orthogonal projection from $\LD$  onto $\AD$. 
For all $f\in \CD$, the Toeplitz operator $T_f\in B(\AD)$ is defined by 
$$
 T_f(\psi):= P(f\hs \psi), \qquad \psi \in \AD \subset \LD, 
$$
and the Toeplitz algebra $\T$ is the C*-algebra generated by all 
$T_f$ in $B(\AD)$. 

It is well known (see e.g.\ \cite{V}) that the compact operators $\KAD$ 
belong to $\T$ and that the quotient $\T / \KAD \cong \CS$ gives rise to 
the C*-algebra extension 
\[ \label{Cex1} 
\xymatrix{
 0\;\ar[r]&\; \KAD\;\ar[r] &\; \T \;\ar[r]^{\s\ \ } &\; \CS\;\ar[r] &\ 0\, ,} 
 \]
where $\s : \T \lra \CS$ is given by $\s(T_f) =  f\!\! \upharpoonright_{\mathbb{S}^1}$ 
for all $f\in \CD$. 

There are alternative descriptions for the Toeplitz algebra. For instance, 
consider the Hilbert space $L_2(\bS)$ with respect to the Lebesgue measure on $\bS$ 
and the orthonormal basis $\{ \frac{1}{\sqrt{2\pi}} u^k:k\in\Z\}$, where $u\in \CS\subset L_2(\bS)$ 
is the unitary function given by $u(\zeta) = \zeta$, \,$\zeta \in \bS$. 
Let $P_+$ denote the orthogonal projection from $L_2(\bS)$ 
onto $ \cspan\{ u^n : n\in \N\} \cong \lN$. For all $f\in \CS$, 
define $\hat T_f\in B(\lN)$ by 
\[ \label{TfN} 
 \hat T_f(\phi):= P_+(f\hs \phi), \qquad \phi \in \cspan\{ u^n : n\in \N\}  \subset L_2(\bS). 
\]
Then $\T$ is isomorphic to the C*-subalgebra of $B(\lN)$ generated by the operators 
$\{\hat T_f : f\in \CS\}$, and the C*-algebra extension \eqref{Cex1} becomes 
\[ \label{Cex2} 
\xymatrix{
 0\;\ar[r]&\; \KlN\;\ar[r] &\; \T \;\ar[r]^{\s\ \ } &\; \CS\;\ar[r] &\ 0\,}
 \]
with $\s(\hat T_f)=f$. 

Let us also mention that $\T$ may be considered as a deformation 
of the C*-al\-ge\-bra of continuous functions 
on the closed unit disc $\bD$ (see \cite{KL}). From this point of view, 
the equivalent C*-algebra extensions \eqref{Cex1} and \eqref{Cex2} correspond to 
the exact sequence 
\[ \label{Cex3} 
\xymatrix{
 0\;\ar[r]&\; C_0(\D)\;\ar[r] &\; C(\bD) \;\ar[r]^{\tau\ \ } &\; \CS\;\ar[r] &\ 0\, ,} 
 \]
where $\tau(f)= f\!\! \upharpoonright_{\mathbb{S}^1}$. 

Recall that the torus $\bT$ can be constructed as a topological manifold by 
dividing the boundary $\bS =\partial \bD$ into four quadrants and gluing 
opposite edges together. Then the C*-algebra of continuous functions on $\bT$ is 
isomorphic to 
\[ \label{CbT}
\CbT:= \{ f\in \CD \hs:\hs   f(\e^{\im t}) \hsp =\hsp f(-\im \e^{-\im t}),\ f(\e^{-\im t}) \hsp=\hsp f(\im\e^{\im t}), \ 
t\hsp\in\hsp [0,\mbox{$\frac{\pi}{2}$}]\}. 
\]
Motivated by \eqref{CbT} and the analogy between \eqref{Cex3}  and \eqref{Cex1} (or \eqref{Cex2}), 
we state the following definition of the noncommutative Toeplitz torus: 
\begin{defn} \label{CTq}
The C*-algebra of the noncommutative Toeplitz torus is defined by 
$$
\CTq\hsp :=
\hsp \{ a\hsp \in\hsp \T :   \s(a)(\e^{\im t}) \!=\! \s(a)(-\im \e^{-\im t}),\ \s(a)(\e^{-\im t}) \!=\! \s(a)(\im\e^{\im t}), \ 
t\!\in\! [0,\hsp\mbox{$\frac{\pi}{2}$}]\}. 
$$
\end{defn}
That $\CTq$ is a C*-subalgebra of $\T$ follows from the fact that $\s$ is a C*-al\-ge\-bra homomorphism. 
Note that gluing the point $\e^{\im t} \in \bS$ to $-\im \e^{-\im t}\in \bS$ 
and the point $\e^{-\im t} \in \bS$ to $\im\e^{\im t} \in\bS$ for all $t\hsp\in\hsp [0,\mbox{$\frac{\pi}{2}$}]$ 
yields a topological space homeomorphic to the wedge sum  $\bS\hsp\vee \bS$ of two pointed circles. 
Setting 
\[ \label{SvS} 
C(\bS\hsp\vee \bS) := \{ f\in\CS : f(\e^{\im t}) \!=\! f(-\im \e^{-\im t}),\ f(\e^{-\im t}) \!=\! f(\im\e^{\im t}), \ 
t\!\in\! [0,\hsp\mbox{$\frac{\pi}{2}$}]\}, 
\]
we can write 
\[ \label{aD}
\CTq = \{ a \in \T : \s(a)\in C(\bS\hsp\vee \bS)\}. 
\]
Moreover, \eqref{Cex2} and \eqref{aD} yield the C*-algebra extension 
$$
\xymatrix{
 0\;\ar[r]&\; \KlN\;\ar[r] &\; \CTq \;\ar[r]^{\s\ \ } &\; C(\bS\hsp\vee \bS)\;\ar[r] &\ 0\,.}
$$

\section{Spectral triple on the noncommutative Toeplitz torus}    \label{sec-1} 

The Dirac operator on a local chart in two dimensions with the flat metric, 
see \cite{F}, up to constant and change of orientation is given by 
\[ \label{Dclass} 
D=  \begin{pmatrix} 0 & \ddz \\ -\ddbz & 0 \end{pmatrix}, \quad  
\ddz = \half\Big( \mbox{$\frac{\partial}{\partial x}-\im \frac{\partial}{\partial y}$}\Big), \ \  
\ddbz=\half \Big(\mbox{$\frac{\partial}{\partial x}+\im \frac{\partial}{\partial y}$}\Big). 
\]
Since $\ddz$ acts on $\AD$ in the obvious way, we want to use the same structure to define a 
spectral triple for the noncommutative Toeplitz torus. 
Clearly, one can construct a noncommutative  version of 
any (orientable) compact surface without boundary by choosing 
appropriate boundary conditions in Definition~\ref{CTq} (see \cite{W}). However, by the 
Gauss--Bonnet theorem, only the (classical) torus admits a dense local chart with a flat metric, 
therefore we restrict here our discussion to the quantum analogue of the torus. 

Our principal aim is to find a dense *-subalgebra $\A\subset \CTq$ and an operator $\dz$, 
which should be closely related to $\ddz$ from \eqref{Dclass}, such that 
$[\dz, a]$ is bounded for all $a\in \A$. Recall that an orthonormal basis for $\AD$ is given by 
$\{\varphi_n \;{:}\; n\hsp\in\hsp\N\}$, where $\varphi_n:= \frac{\sqrt{n+1}}{\sqrt{\pi}} \hs z^n$ \cite{V}. 
Complex differentiation  yields $\ddz(\varphi_n) = \sqrt{n(n+1)} \hs \varphi_{n-1}$. 
If we define an operator $\dz$ on $\AD$ by $\dz(\varphi_n):= n\hs \varphi_{n-1}$, then 
$\ddz - \dz$ extends to a bounded operator on $\AD$ since the coefficients $ \sqrt{n(n+1)} -n$ 
are uniformly bounded. As a consequence, the commutators 
$[\ddz, a]$ are bounded for all $a\in\A$ if and only if the commutators with $\dz$ are bounded. 

In order to simplify the notation, we will use the description of the 
Toeplitz algebra on $\lN \cong \cspan\{ u^n : n\in \N\} \subset L_2(\bS)$. For $m\in\Z$, 
set $e_m:= \frac{1}{\sqrt{2\pi}} u^m$ and let $\dz$ be defined by 
\[ \label{dz}
\dz(e_n) := n \hs e_{n-1}\ \, \text{on} 
\ \, \dom(\dz) := \big\{ \msum{n\in\N}{} \a_n  e_n\in \lN:  \msum{n\in\N}{}  n^2 |\a_n|^2<\infty\big\}. 
\]
Moreover, consider the number operator $N$ on $\lN$ determined by 
\[ \label{N}
N(e_n) := n \hs e_{n}\quad  \text{on} 
\quad \dom(N) := \dom(\dz). 
\]
Let $S$ be the unilateral shift operator on $\lN$ so that we have 
\[ \label{S}
S(e_n) = e_{n+1}, \ \ n\in\N,  \quad S^*(e_n) = e_{n-1}, \ \ n>1, \quad S^*(e_0)=0. 
\]
Since $N$ is a self-adjoint positive operator on $\dom(N) = \dom(\dz)$ and since 
$S^*$ is a partial isometry such that $\ker(S^*)=\mathrm{Ran}(N)^\bot$, 
it follows that $\dz=S^*N$ is the polar decomposition of the closed operator $\dz$. 
Clearly, $\dzs=N S$, so 
\[ \label{dzs} 
\dzs(e_n) = (n+1) \hs e_{n+1}\quad  \text{and} \quad \dom(\dzs)= \dom(N). 
\] 

Under the unitary isomorphism $\AD \cong \lN$ given by $\varphi_n \mapsto e_n$ 
on the bases described above, the operator $\dz$ on $\lN$ is unitary equivalent to a bounded perturbation 
of the Cauchy-Riemann operator $\ddz$ on $\AD$. 
Therefore we take $\dz$ on $\lN$ as a replacement for $\ddz$ on $\AD$. 

Note that, in the commutative case and with functions represented by multiplication operators, 
one has $[\ddz, f] = \frac{\partial f}{\partial z}$ for all $f\in C^{(1)}(\D)$
but clearly not all 
continuous functions are differentiable. In the following, we will single out a dense 
*-subalgebra $\A\subset \CTq\subset B(\lN)$ which can be viewed as an algebra of 
infinitely differentiable functions. 
With $C(\bS\vee\bS)\subset \CS$ defined in \eqref{SvS}, set 
$C^{\infty}(\bS\vee\bS):= C(\bS\vee\bS)\cap C^{\infty}(\bS)$ and let 
$$
\A_0:= \{ \hat T_f : f\in C^{\infty}(\bS\vee\bS)\} \subset \CTq. 
$$
Using the obvious embedding $\End(\lin\{e_1,\ldots,e_n\})\subset K(\lN)\subset \CTq$, consider 
$$
\F_0:= \bigcup_{n\in\N} \hs \End(\lin\{e_1,\ldots,e_n\})\subset \CTq. 
$$
We will take $\A$ to be the *-subalgebra of $\CTq$ generated by the elements of $\A_0$ and $\F_0$, 
i.e., 
\[ \label{A}
\A:= \text{{\rm *-alg}}(\A_0 \hs\cup\hs \F_0)\subset \CTq. 
\] 
\begin{lem} \label{lem}
The algebra $\A$ defined in \eqref{A} is dense in $\CTq$ 
and its elements admit bounded commutators with $\dz$ and $\dzs$. 
Furthermore, $\delta_N^k(a)$, $\delta_N^k([\dz,a])$ and $\delta_N^k([\dzs,a])$ are 
bounded for all $a\in\A$ and $k\in \N$, where $\delta_N(x):=[N,x]$ for $x\in B(\lN)$. 
\end{lem} 
\begin{proof}
The set $\F_0$ contains all finite operators on $\lin\{e_n:n\in\N\}$, 
therefore it is dense in $\KlN$. As a consequence, all compact operators $\KlN$ belong to 
the closure of~$\A$. 
From \eqref{TfN}, it follows that $\|\hat T_f\|\leq \|f\|_\infty$. 
By the Stone--Weierstrass theorem, $C^{\infty}(\bS\vee\bS)$ is dense in $C(\bS\vee\bS)$ 
with respect to the norm $\|\cdot\|_\infty$.  
Thus each $\hat T_g\in \CTq$ can be approximated by elements from $\A_0$. Let $a\in \CTq$. 
Writing $a = a-\hat T_{\s(a)} + \hat T_{\s(a)}$, where 
$\hat T_{\s(a)} \hsp\in\hsp  \CTq$ and $a-\hat T_{\s(a)}\hsp\in\hsp\KlN$, we conclude that $a$ lies in the 
clo\-sure of $\A$, so $\A$ is dense in $\CTq$.

By the Leibniz rule $[A, BC]=[A, B]C+B[A, C]$ for the commutator $[\cdot\hs ,\cdot]$, 
it suffices to prove the boundedness of the commutators for the elements belonging to 
the generating set $\A_0 \hs\cup\hs \F_0$. 
From the definitions of $\F_0$ and $N$, it follows that 
$Na\in \F_0$ and $aN\in \F_0$ for all $a\in\F_0$. This immediately that 
implies $\delta_N^k(a)\in B(\lN)$ for all $k\in \N$ since each term
of the iterated commutators belongs to $\F_0\subset B(\lN)$. 
Note also that $aS^*\in \F_0$ and  $S^*a\in \F_0$ for all $a\in\F_0$, 
therefore
$[\dz, a] =  S^* (N a) - (a S^*)N \in \F_0$. 
In particular,  $[\dz, a]$ and $\delta_N^k([\dz, a])$ are bounded 
for all $k\in\N$. 

Next  consider $\hat T_f\in \A_0$. To determine the action of $\hat T_f$ on $\lN$, 
we represent $f$ by its Fourier series 
$f = \sum_{k\in\Z} \hat f(k) \hs u^k$, where $\hat f(k)\in\C$. 
Since multiplication by $u^k$ yields $u^k  e_m = e_{m+k}$, one obtains from \eqref{TfN} 
\[ \label{Tfrep}
\hat T_f(e_m) = P_+\Big(\msum{k\in\Z}{} \hat f(k) \hs u^k\hs e_m\Big) = 
P_+\Big(\msum{k\in\Z}{} \hat f(k) \hs  e_{m+k}\Big) = \msum{n\in\N}{} \hat f(n\hsp-\hsp m)\hs e_n\hs.
\]
If $f\in C^{\infty}(\bS)$, then partial integration shows that  $f'\in \CS$ has the Fourier series 
$f' = \sum_{k\in\Z}  \im k \hat f(k) \hs u^k$. 
Therefore, for all $m\in\N$, 
\begin{multline}
[N, \hat T_f](e_m) = 
 \msum{n\in\N}{} n \hat f(n\hsp-\hsp m)\hs e_{n} -   \msum{n\in\N}{} m \hat f(n\hsp-\hsp m)\hs e_{n} 
 =   \msum{n\in\N}{} (n\hsp-\hsp m) \hat f(n\hsp-\hsp m)\hs e_{n}\\
 = -\im \hs P_+\Big( \msum{k\in\Z}{} \im (k\hsp-\hsp m) \hat f(k\hsp-\hsp m)\hs e_{k} \Big)  
 = -\im \hs \hat T_{f'}(e_m)  \label{Tf} 
\end{multline} 
by \eqref{Tfrep} and the Fourier series of $f'$. Similarly, 
\begin{align} 
[\dz, \hat T_f](e_m) &= 
 \msum{n\in\N}{} n \hat f(n\hsp-\hsp m)\hs e_{n-1} -   \msum{n\in\N}{} m \hat f(n\hsp-\hsp (m\hsp -\hsp 1))\hs e_{n} \nonumber \\
 &=   \msum{n\in\N}{} (n\hsp-\hsp m+1) \hat f(n\hsp-\hsp m+1)\hs e_{n} 
 = -\im \hs P_+\Big(\bar u \msum{k\in\Z}{} \im (k\hsp-\hsp m) \hat f(k\hsp-\hsp m)\hs e_{k} \Big)  \nonumber \\
  &= -\im \hs \hat T_{\bar u f'}(e_m).  \label{Tuf} 
\end{align} 
This yields 
$[\dz, \hat T_f] =  -\im \hs \hat T_{\bar u f'}\in B(\lN)$, 
\,$\delta_N^k(\hat T_f) = (-\im)^k \hs  \hat T_{f^{\hspace{-0.5pt}(k)}} \in B(\lN)$,  
and $\delta_N^k([\dz, \hat T_f]) = (-\im)^{k+1} \hs  \hat T_{(\bar uf')^{\hspace{-0.5pt}(k)}} \in B(\lN)$, 
the latter because $\bar uf'$ is a $C^{\infty}$-function. 
The statement for $\dzs$ can be proven analogously or by using 
$[\dzs, a]= -[\dz, a^*]^*$ together with 
$a^*\in \F_0$ for all $a\in\F_0$  and $\hat T_f^* = \hat T_{\bar f}$ for all $f\in\CS$. 
\end{proof} 

Now we are in a position to construct our spectral triple and describe its fundamental properties. 
\begin{thm} 
Let $\A$ denote the dense *-subalgebra of \hs$\CTq$ from Lemma \ref{lem}. 
Set $\hH:= \lN\oplus\lN$ and
define a *-representation $\pi: \A\ra B(\lN\hsp\oplus\hsp\lN)$  by 
$\pi(a):= a\oplus a$.  Consider the self-adjoint operator 
$$
D:=  \begin{pmatrix} 0 & \dz \\ \dzs & 0 \end{pmatrix} \quad \text{on} \quad 
\dom(D):= \dom(N) \oplus \dom(N). 
$$
Then $(\A,\hH,D)$ is a $1^+$-summable regular even spectral triple for $\CTq$ 
with grading operator $\gamma:= \id \oplus (-\id)$.  
The Dirac operator $D$ has discrete spectrum $\spec(D) =\Z$, 
each eigenvalue $k \in \spec(D)$ has multiplicity~1, and 
a complete set of eigenvectors $\{b_k:k\in\Z\}$ satisfying $Db_k=k\hs b_k$ 
is given by 
$$
b_k:= \mbox{$\frac{1}{\sqrt{2}}$}(e_{k-1} \oplus e_k), \quad 
b_{-k}:= \mbox{$\frac{1}{\sqrt{2}}$}(- e_{k-1} \oplus e_k), \quad  k>0, \qquad b_0:= 0\oplus e_0. 
$$ 
Its fundamental class $F= \begin{pmatrix} 0 & S^* \\ S& 0 \end{pmatrix}$ 
is non-trivial and $\ind(D)=1$. 
\end{thm} 
\begin{proof}
We have already mentioned that the operator $\dz=S^*N$ is closed. 
Hence $D$ is self-adjoint by its definition. 
Since $[D,\pi(a)]$ has $[\dz,a]$ and $[\dzs,a]$ as its non-zero matrix entries, 
the boundedness of these commutators for all $a\in\A$ follows from Lemma~\ref{lem}. 
As $\dz =S^*N$  and $\dzs= NS= S(N\hsp +\hsp 1)$, the polar decomposition of $D$ reads as 
\[ \label{PD}
D= F\hs |D| =  \begin{pmatrix} 0 & S^* \\ S& 0 \end{pmatrix} \begin{pmatrix} N+1 & 0 \\ 0& N \end{pmatrix} . 
\]
In particular, the entries of the commutators with $|D|$ are given by commutators with $N$, thus 
the regularity can easily be deduced from Lemma~\ref{lem}. 
Clearly,  $\gamma$, $D$ and $\pi(a)$ satisfy \eqref{gam}, so the spectral triple is even. 
From \eqref{dz} and \eqref{dzs}, it follows immediately that $D(b_k)=k\hs b_k$ 
for all $k\in \Z$. Since $\{b_k: k\in \Z\}$ is an orthonormal basis for $\hH$, we have 
$\spec(D) =\Z$ and each eigenvalue has multiplicity 1. The $1^+$-summability 
follows from the convergence behavior of the series $\sum_{k\in\Z} (1+|k|)^{-(1+\epsilon)}$, 
$\epsilon \geq 0$. Finally, by the polar decomposition given in \eqref{PD}, $\ind(D)=\ind(S^*)=1$. 
\end{proof} 

\section*{Acknowledgements} 
This work was partially supported by CIC-UMSNH and 
the Polish Government grant 3542/H2020/2016/2, and  
and is part of the project supported by the EU funded grant H2020-MSCA-RISE-2015-691246-QUANTUM DYNAMICS.

\end{document}